\documentclass[a4paper]{amsart}
\usepackage[utf8]{inputenc}
\usepackage[T1]{fontenc}
\usepackage{mathrsfs}   
\usepackage{xfrac,amssymb}
\usepackage{todonotes}
\usepackage{url}
\usepackage{tikz}
\usetikzlibrary{calc,matrix,decorations.pathreplacing,svg.path}
\newcommand{\FF}{\mathbb{F}}
\newcommand{\NN}{\mathbb{N}}
\newcommand{\QQ}{\mathbb{Q}}
\newcommand{\RR}{\mathbb{R}}
\newcommand{\ZZ}{\mathbb{Z}}
\newcommand{\setstyle}{\mathscr}
\newcommand{\setA}{\setstyle{A}}
\newcommand{\setC}{\setstyle{C}}
\newcommand{\setS}{\setstyle{S}}
\newcommand{\idealP}{\mathfrak{p}}
\newcommand{\fieldK}{K}
\newcommand{\fieldL}{L}
\newcommand{\intring}{\ZZ_\fieldK}
\newcommand{\valring}{\mathcal{O}_\idealP}
\newcommand{\polyring}{\intring[x]}
\newcommand{\fieldKp}{\fieldK_\idealP}
\newcommand{\CK}{\setC(\fieldK)}
\newcommand{\CKp}{\setC(\fieldKp)}
\newcommand{\CZ}{\setC(\intring)}
\newcommand{\COp}{\setC(\valring)}
\newcommand{\un}[1][\fieldKp]{#1^\times}
\newcommand{\unif}{\pi}
\newcommand{\pow}{\fieldKp^p}
\newcommand{\unpow}{\fieldKp^{\times p}}
\newcommand{\pf}{F}
\newcommand{\pg}{G}
\newcommand{\ph}{H}
\newcommand{\st}{\;|\;}
\newcommand{\quo}[2]{#1/#2}
\newcommand{\card}[1]{|#1|}
\DeclareMathOperator{\kras}{kras}
\DeclareMathOperator{\lc}{lc}
\DeclareMathOperator{\ord}{ord}
\newcommand{\ordp}{\ord_\idealP}
\DeclareMathOperator{\rev}{rev}
\newcommand{\onto}{\twoheadrightarrow}
\newcommand{\term}[1]{\emph{#1}}
\newtheorem{thm}{Theorem}[section]
\newtheorem{alg}[thm]{Algorithm}
\newtheorem{cor}[thm]{Corollary}
\newtheorem{lem}[thm]{Lemma}
\newtheorem{obs}[thm]{Observation}
\newtheorem{prop}[thm]{Proposition}
\theoremstyle{remark}
\newtheorem{rem}[thm]{Remark}
\newenvironment{poc}{\begin{proof}[Proof of correctness]}{\end{proof}}
\title{Polynomials assuming only local prime powers}
\author{Przemysław Koprowski}
\begin{document}
\begin{abstract}
We study the class of polynomials that map a local field (i.e., the completion of a number field at a non-Archimedean place) into the subset of its $p$-th powers, where $p$ is the residue characteristic of the field in question. We present a characterization of such polynomials and show that this class is always much broader than the class of $p$-th powers of polynomials.
\end{abstract}
\maketitle

\section{Introduction}
It is a classical result (see, e.g., \cite[Problem~114]{PS1998}) that if $\pf\in \ZZ[x]$ is a polynomial with integer coefficients such that $\pf(a)$ is a square for every rational argument $a \in \QQ$, then $\pf$ itself must be the square of another polynomial. There is also a multivariate generalization of this theorem due to Murty (see \cite{Murty2008}).

This raises the natural question of whether similar statements remain valid over other fields occurring naturally in algebraic number theory. In a private conversation, Carlos D’Andrea asked the author if the above result extends to the setting of $2$-adic numbers. More precisely: does there exist a polynomial $\pf\in \ZZ[x]$ such that $\pf(a)$ is a $2$-adic square for every rational number $a \in \QQ$, while $\pf$ itself is not the square of any polynomial with $2$-adic coefficients? An example of such a polynomial is $4x^4 + 4x^2 + 9$. We omit the proof here, since a more general result will be established later (see Proposition~\ref{prop:non_power_poly}). Nonetheless, this question motivated a systematic study of a class of polynomials--- denoted~$\CKp$---with coefficients in rings of algebraic integers, whose values are always local prime powers, where the relevant exponents coincide with the residue characteristic of the local field. A precise definition of the class $\CKp$ is given in Section~\ref{sec:notation}.

The class $\CKp$ has an interesting geometric interpretation. Recall (see, e.g., \cite{FV2002}) that if $\fieldKp$ is a local field (i.e., the completion of a number field with respect to a non-Archimedean valuation) with residue characteristic~$p$, then $\fieldKp$ can be partitioned into finitely many disjoint open sets of the form $a\unpow$ (so they are classes of $p$-th powers of nonzero elements of $\fieldKp$), together with the singleton $\{0\}$, which serves as their common boundary point. Thus, $\fieldKp$ can be visualized as a `fan' of homeomorphic sets of the form $a\unpow$, joined together at zero (see Figure~\ref{fig:Q3}). A polynomial~$\pf$ lies in~$\CKp$ precisely when it collapses this `fan' structure into a single `blade', mapping the entire local field~$\fieldKp$ into the set of $p$-th powers. One can think of the class~$\CKp$ as---to some extent---analogous to the set of real polynomials that are non-negative everywhere. The field~$\RR$ decomposes into two open half-lines, $-1\RR_{+}$ and $1\RR_{+}$, together with their common boundary point $\{0\}$. A polynomial that is nonnegative everywhere can be seen as folding these two half-lines onto one another.

In this paper, we study the class~$\CKp$ and several related classes. The paper is organized as follows. In Section~\ref{sec:notation} we introduce the relevant notation and, in particular, define the class $\setC(\setA)$ for any subset~$\setA$ of~$\fieldKp$. Section~\ref{sec:preliminaries} presents a characterization of the polynomials that map the entire local field~$\fieldKp$ into its set of $p$-th powers (see Proposition~\ref{prop:CK_iff_F_and_revF}). The class~$\CKp$ is decidable by the classical Ax–Kochen theorem; in Section~\ref{sec:decidability} we give a simple explicit algorithm (Algorithm~\ref{alg:test_CK}) for testing membership in this class. Finally, in Section~\ref{sec:Pth_powers} we return to the motivating question and show (see Proposition~\ref{prop:non_power_poly}) that the class of polynomials mapping~$\fieldKp$ into its set of $p$-th powers is always strictly larger than the class of $p$-th powers of polynomials—so the analogue of the classical result mentioned at the beginning never holds over local fields.

In this paper, we study the class~$\CKp$ and some of its related classes. The paper is organized as follows. In Section~\ref{sec:notation} we introduce all the relevant notation. In particular, we define the class $\setC(\setA)$ for any subset~$\setA$ of~$\fieldKp$. Next, in Section~\ref{sec:preliminaries} we present the characterization of the polynomials that map the entire local field~$\fieldKp$ into the set of its $p$-th powers (see Proposition~\ref{prop:CK_iff_F_and_revF}). The class~$\CKp$ is decidable by the classical Ax–Kochen theorem; in Section~\ref{sec:decidability} we give a simple explicit algorithm (Algorithm~\ref{alg:test_CK}) for testing membership in this class. Finally, in Section~\ref{sec:Pth_powers} we return to the motivating question and show (see Proposition~\ref{prop:non_power_poly}) that the class of polynomials mapping~$\fieldKp$ into its set of $p$-th powers is always strictly larger than the class of $p$-th powers of polynomials—so the analogue of the classical result mentioned at the beginning never holds over local fields.

%

%
%
%
\section{Notation}\label{sec:notation}
Throughout this paper, we use the following notation. Let $\fieldK$ be a fixed number field and let $\intring$ denote its ring of integers (i.e., the integral closure of~$\ZZ$ in~$\fieldK$). Next, let $\idealP$ be a prime of~$\fieldK$ lying above a rational prime~$p$. By $\fieldKp$ we denote the completion of~$\fieldK$ at~$\idealP$. We write $\ordp$ for the unique normalized valuation $\fieldKp \to \ZZ$, as well as its restriction to~$\fieldK$. The associated valuation ring is $\valring = \{ a\in \fieldKp\st \ordp a\geq 0\}$. ts maximal ideal is~$\idealP$, and the quotient $\quo{\valring}{\idealP}$ is a finite field, called the \term{residue class field} of~$\idealP$. An element of valuation one is called a \term{uniformizer} of~$\idealP$. By the strong approximation theorem, one can always choose a uniformizer lying in~$\intring$. It will be denoted~$\unif$ throughout the paper. Finally, we use the standard notation $e = e(\idealP \mid p)$ and $f = f(\idealP \mid p)$ for the \term{ramification index} (i.e., the valuation of~$p$) and the \term{inertia degree} (i.e., the degree of the residue class field over its prime field~$\FF_p$), respectively. The quotient group $\quo{\un}{\unpow}$ of the multiplicative group of~$\fieldKp$ modulo the $p$-th powers is an elementary $p$-group. By Artin's theorem (see, e.g., \cite[Theorem~2.36]{Clark2010} or \cite[Proposition~II.6]{Lang1994}) its order is precisely $p^{ef + 2}$ if $\fieldKp$ contains a primitive $p$-th root of unity, and $p^{ef + 1}$ otherwise.

As explained in the introduction, our main interest lies in polynomials whose values are local $p$-th powers, where $p$ is the residue characteristic of the considered local field. For any nonempty subset $\setA \subseteq \fieldKp$, we define
\[
\setC(\setA) := 
\bigl\{ \pf\in \polyring \st \pf(a)\in \pow\text{ for every }a\in \setA\bigr\}.
\]
We are particularly interested in the cases when $\setA$ is one of the rings $\intring$, $\valring$, $\fieldK$, or $\fieldKp$.

\begin{figure}
\begin{tikzpicture}[scale=3]
  \newcommand{\qqq}{\QQ_3^{\times 3}}
  \def\labels{{ "$\qqq$", "$2\cdot \qqq$", "$3\cdot \qqq$", "$4\cdot \qqq$", "$6\cdot \qqq$", 
              "$9\cdot \qqq$", "$12\cdot \qqq$", "$18\cdot \qqq$", "$36\cdot \qqq$" }}

\foreach \k in {0,...,8} {
  \begin{scope}[rotate=-\k*40]
    \ifnum\k=0
      \path[draw, fill=gray!30, dotted] svg "M 0,0 C 9,28 10,40 0,40 -10,40 -9,28 0,0 Z"; 
      \node at (0,1) {$\qqq$};
    \else
      \path[draw, fill=gray!5, dotted] svg "M 0,0 C 9,28 10,40 0,40 -10,40 -9,28 0,0 Z";
      \node at (0,1) {\strut\pgfmathprint{\labels[\k]}};      
    \fi
  \end{scope}
}
  \fill (0,0) circle (1.5pt);
  \node[below left] at (0,0) {\large $0$};

\end{tikzpicture}
\caption{\label{fig:Q3}The field $\QQ_3$ as a union of classes of third powers.}
\end{figure}
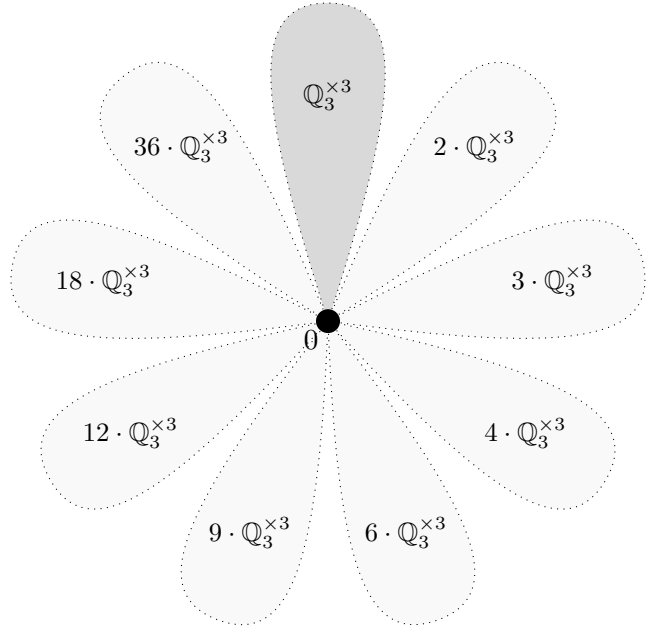

\section{Preliminaries}\label{sec:preliminaries}
The following proposition and its corollary will be used repeatedly throughout the paper. Although they are classical results, we state them explicitly for the reader’s convenience.

\begin{prop}[{\cite[Corollary~I.5.8.2]{FV2002}}]\label{prop:local_power_thm}
For every $k> \sfrac{ep}{(p-1)}$ one has $1 + \idealP^k\subseteq \unpow$.
\end{prop}

\begin{cor}\label{cor:equal_classes_suff}
Let $a, b \in \valring$ be two elements of equal valuation $v := \ordp(a) = \ordp(b)$. If
\[
a \equiv b \pmod{\idealP^k}
\qquad\text{for some } k > v + \frac{ep}{p-1},
\]
then the cosets $a \unpow$ and $b \unpow$ coincide.
\end{cor}

As mentioned at the end of the previous section, we focus on the relations between the classes $\CZ$, $\COp$, $\CK$, and $\CKp$. These classes satisfy the natural inclusions
\[
\begin{tikzpicture}
  \matrix (m) [matrix of math nodes,
               row sep=1em, column sep=2em,
               text height=1.5ex, text depth=0.25ex]
  {
    \CK & \CZ \\
    \CKp & \phantom{.}\COp. \\
  };

  \node at ($(m-1-1)!0.5!(m-1-2)$) {$\subseteq$};
  \node at ($(m-2-1)!0.5!(m-2-2)$) {$\subseteq$};

  \node at ($(m-1-1)!0.5!(m-2-1)$) {\rotatebox{90}{$\subseteq$}};
  \node at ($(m-1-2)!0.5!(m-2-2)$) {\rotatebox{90}{$\subseteq$}};
\end{tikzpicture}
\]
We will prove that in the above diagram, the vertical inclusions are in fact equalities, while the horizontal inclusions are strict. Let us begin with the vertical ones.

\begin{prop}\label{prop:CK_eq_CKp}
The following equalities hold:
\[
\CK = \CKp\qquad\text{and}\qquad \CZ = \COp. 
\]
\end{prop}

\begin{proof}
We prove only the first statement. The other one is fully analogous. One inclusion is immediate and was already observed above. It remains to show that $\CK \subseteq \CKp$. The key ingredients are the density of $\fieldK$ in $\fieldKp$ and the fact that $\un$ decomposes into a finite disjoint union of open subsets.

Fix a nonzero element $a \in \fieldKp$. Since $\fieldK$ is dense in~$\fieldKp$, there exists a sequence $(a_i)_{i \in \NN} \subseteq \fieldK$ converging to~$a$ in the $\idealP$-adic topology. A contrario, suppose that $\pf(a)$ is not a $p$-th power in~$\fieldKp$. In particular, $\pf(a) \neq 0$. Write
\[
\pf(a) = \unif^v \cdot u,
\]
where $v := \ordp \pf(a)$, $\unif \in \intring$ is a uniformizer, and $u$ is a $\idealP$-unit. Since polynomials are continuous with respect to the $\idealP$-adic topology, we have
\[
\pf(a) = \lim_{i \to \infty} \pf(a_i).
\]
Therefore, there is an index~$i$ such that
\[
\ordp \bigl( \pf(a) - \pf(a_i)\bigr) > v + \frac{ep}{p - 1}.
\]
Corollary~\ref{cor:equal_classes_suff} asserts that $\pf(a)$ and $\pf(a_i)$ belong to the same $p$-th power class. But this contradicts our assumption: $\pf(a_i) \in \pow$ while $\pf(a) \notin \pow$. This contradiction proves the claim.
\end{proof}

Before establishing the relationship between $\CZ$ and $\CK$, we need to prove some auxiliary results.

\begin{lem}\label{lem:p_div_deg_and_ord}
If $\pf\in \CK$, then both the degree of~$\pf$ and the $\idealP$-adic valuation of its leading coefficient are divisible by~$p$.
\end{lem}

\begin{proof}
Let $\pf = \pf_0 + \pf_1x + \dotsb + \pf_dx^d$, with $\pf_d \neq 0$. A contrario, suppose that either $d$ or $\ordp \pf_d$ is not a multiple of~$p$. Then there is an integer $k\in \ZZ$, possibly negative, such that $p$ does not divide the expression $kd + \ordp \pf_d$ and
\[
k < \min \Bigl\{
\bigl\lfloor \ordp (\pf_i/\pf_d)^{\sfrac{1}{(d - i)}} \bigr\rfloor
\st
i < d\Bigr\}.
\]
Let $\unif\in \intring$ be a uniformizer of~$\idealP$. It follows from the above inequality that
\[
\ordp\!\bigl(\pf_d \unif^{kd}\bigr) < \ordp\!\bigl(\pf_i \unif^{ki}\bigr)
\qquad\text{for all}\ i \in \{0,\dotsc,d-1\}.
\]
Consequently, we have 
\[
\ordp \pf(\unif^k) 
= \ordp \bigl( \pf_d\cdot \unif^{kd}\bigr)
= kd + \ordp \pf_d.
\]
By construction, the right-hand side is not divisible by~$p$. Hence, $\pf(\unif^k)$ cannot be a $p$-th power in~$\fieldKp$, contradicting the assumption $\pf\in \CK$. This completes the proof.
\end{proof}

Let $\pf = \pf_0 + \pf_1x + \dotsb + \pf_dx^d$ be a polynomial of degree~$d$. Recall that the \term{reciprocal polynomial} $\rev\pf$ is defined by the formula 
\[
\rev\pf = x^d\cdot \pf\Bigl(\frac{1}{x}\Bigr)
= \pf_0x^d + \pf_1x^{d - 1} + \dotsb + \pf_d.
\]

\begin{lem}\label{lem:\pf_if\pf_revf}
A polynomial $\pf$ belongs to $\CK$ if and only if its reciprocal $\rev\pf$ does.
\end{lem}

\begin{proof}
It suffices to prove one implication. Assume that $\pf\in \CK$, and set $d = \deg\pf$. For any nonzero element $a \in \fieldK$ we have
\[
(\rev\pf)(a) = a^d \cdot \pf\Bigl(\frac{1}{a}\Bigr).
\]
Since $\pf\in \CK$, we know $\pf(\sfrac{1}{a}) \in \pow$. By the previous lemma, $p$ divides~$d$, so $a^d \in \pow$ as well. Hence, $(\rev\pf)(a)$ is a $p$-th power in~$\fieldKp$ for all nonzero~$a$.

It remains to show that $\pf_d(a)$ is a local $p$-th power also when $a$ is zero. Here $(\rev\pf)(0) = \pf_d$ is the leading coefficient of~$\pf$. By Lemma~\ref{lem:p_div_deg_and_ord}, $\ordp(\pf_d)$ is divisible by~$p$, so we may write
\[
\pf_d = \unif^{kp} \cdot u
\]
for some $k \geq 0$ and some $\idealP$-adic unit $u \in \fieldK$. Evaluate $\rev\pf$ at $\unif^m$ for an integer $m > kp + \sfrac{ep}{p-1}$. The previous part of the proof asserts that it is some $p$-th power in~$\fieldKp$. Say,
\[
(\rev\pf)(\unif^m) = c^p \qquad \text{for some } c \in \fieldKp.
\]
Thus, we have
\[
c^p = \unif^{kp}\cdot u + \pf_{d-1}\unif^m + \dotsb + \pf_0 \unif^{dm}.
\]
It follows that $u$ is congruent to some $p$-th power modulo $\idealP^n$, with $n = m - kp > \sfrac{ep}{(p-1)}$. By Proposition~\ref{prop:local_power_thm}, $u$ is itself a $p$-th power, and therefore so is $\pf_d$. We conclude that $\rev\pf \in \CK$, completing the proof.
\end{proof}

It is clear that the constant term of any polynomial from $\CK$ must be a $p$-th power in~$\fieldKp$. So is its leading coefficient in view of the previous lemma.

\begin{cor}\label{cor:const_and_lc_in_pow}
If $\pf\in \CK$, then both its constant term and its leading coefficient belong to~$\pow$.
\end{cor}

The next proposition establishes a relation between $\CK$ and $\CZ$.

\begin{prop}\label{prop:CK_iff_F_and_revF}
For any polynomial $\pf\in \polyring$, we have
\[
\pf\in \CK \iff \pf,\; \rev\pf \in \CZ.
\]
\end{prop}

\begin{proof}
If $\pf\in \CK$ then also $\rev \pf\in \CK$ by Lemma~\ref{lem:\pf_if\pf_revf}. Hence, one implication follows immediately from the fact that $\CK \subseteq \CZ$.

For the converse, assume that both~$\pf$ and $\rev\pf$ are in $\CZ$. Fix any element $a\in \fieldKp$. As $\valring$ is the valuation ring of~$\fieldKp$, either $a \in \valring$ or $a^{-1} \in \valring$. If $a\in \valring$ then $\pf(a)\in \pow$ since $\pf\in \CZ = \COp$. 

On the other hand, if $a^{-1}\in \valring$, then we have 
\[
\pf(a) = a^{-d}\cdot (\rev\pf)(a),
\]
where $d = \deg\pf$. Since $\rev\pf \in \CZ = \COp$, we know $(\rev\pf)(a) \in \pow$. Moreover, Lemma~\ref{lem:p_div_deg_and_ord} implies that $p$ divides the degree of~$\pf$, so $a^{-d}$ is itself a $p$-th power in~$\fieldKp$. Thus, $\pf(a) \in \pow$, as well.

In either case, $\pf(a)$ is a $p$-th power, and since $a$ was arbitrary, $\pf\in \CK$. This completes the proof.
\end{proof}

As announced at the beginning of this section, we now show that the inclusion $\CK \subset \CZ$ is strict.

\begin{obs}
The set~$\CK$ is a proper subset of~$\CZ$.
\end{obs}

\begin{proof}
Fix an integer~$m$ such that 
\[
m\equiv 1\pmod{p}
\qquad\text{and}\qquad
m > \frac{ep}{p - 1}.
\]
Consider the polynomial $\pf := 1 + \unif^m\cdot x$. Then $p$ divides neither the degree of~$\pf$ nor the valuation of its leading coefficient. It follows from Lemma~\ref{lem:p_div_deg_and_ord} that $\pf \notin \CK$.

On the other hand, for every $a \in \intring$ we have
\[
\pf(a) \in 1 + \idealP^m \subseteq \unpow.
\]
Hence, $\pf\in \CZ$ by Proposition~\ref{prop:local_power_thm}. This proves that $\CK$ is a proper subset of~$\CZ$.
\end{proof}

\begin{rem}
As a by-product of the above proof, we see that Lemmas~\ref{lem:p_div_deg_and_ord}--\ref{lem:\pf_if\pf_revf} and Corollary~\ref{cor:const_and_lc_in_pow} do not generalize to the class~$\CZ$.
\end{rem}

\section{Decidability of the class $\CKp$}\label{sec:decidability}
The condition that a polynomial~$\pf$ takes $p$-th power values on the whole ring~$\valring$ or the whole field~$\fieldKp$ (i.e., that it belongs to one of the classes $\COp$ or $\CKp$) involves infinitely many, in fact uncountably many, conditions:
\begin{equation}\label{eq:quantified_conditions}
\forall_{a\in \valring}\exists_{b\in\fieldKp} \pf(a) = b^p 
\end{equation}
and analogously for~$\fieldKp$. Proposition~\ref{prop:CK_eq_CKp} from the previous section allows us to reduce the problem to a still infinite but countable family of conditions. On the other hand, \eqref{eq:quantified_conditions} is a first-order formula, and it is known (see, e.g., \cite{AK1966, Kartas2024, Koenigsmann2018, Kruckman2013, Macintyre1976}) that the first-order theory of local fields is decidable. Hence, there must exist a finite family of conditions determining whether a given polynomial belongs to~$\COp$ (resp.~to~$\CKp$). It is possible---at least in theory---to obtain such a family by applying general quantifier-elimination techniques for local fields.. Nonetheless, in this section, we present a simple, explicit algorithmic solution tailored specifically to our setting.

Fix a nonzero polynomial $\pf\in \polyring$. By performing the square-free factorization (see, e.g., \cite[\S 14.6]{GG2013}), we can write
\[
\pf = \lc(\pf)\cdot \pg_1\cdot \pg_2^2\dotsm \pg_k^k,
\]
where $\lc(\pf) \in \intring$ is the leading coefficient of~$\pf$, and $\pg_1, \dotsc, \pg_k$ are monic, square-free polynomials. In general, the coefficients of the $\pg_i$ may lie in~$\fieldK$ rather than~$\intring$. However, every coefficient of each $\pg_i$ can be expressed as $\sfrac{a_{ij}}{c_{ij}}$ with $a_{ij}\in \intring$ and $c_{ij}\in \ZZ$. Multiplying each $\pg_i$ by the $p$-th power of the least common multiple of the denominators, we obtain
\[
c^p\cdot \pf = \lc(\pf)\cdot \tilde\pg_1\cdot \tilde\pg_2^2\dotsm \tilde\pg_k^k,
\]
where the polynomials $\tilde\pg_1, \dotsc, \tilde\pg_k$ now have coefficients in~$\intring$, and $c \in \ZZ$. It follows that $\pf\in \CZ$ if and only if $\pf_*\in \CZ$, where
\[
\pf_* := \lc(\pf) \cdot \prod_{i=1}^k \tilde\pg_i^{i \bmod p}.
\]
Hence, in our analysis, we may restrict to polynomials not divisible by $p$-th powers of other polynomials.

First, we shall deal with linear factors of the polynomials in question.

\begin{lem}\label{lem:root_multiplicity}
Let $\pf \in \polyring$ and suppose $\pf \in \CKp$.  If $\xi \in \fieldKp$ is a root of~$\pf$, then $p$ divides its multiplicity. In particular, if $\pf$ is not divisible by a $p$-th power of another polynomial, then it cannot have roots in~$\fieldKp$.
\end{lem}

\begin{proof}
Write $\pf = (x - \xi)^m \cdot \pg$, where $\pg \in \fieldKp[x]$ is a polynomial not vanishing at~$\xi$. Set $v := \ordp \pg(\xi)$. Suppose the multiplicity~$m$ is not divisible by~$p$. In particular, $\gcd(m, p) = 1$ since $p$ is a prime. Hence, there is an integer $w > v$ such that 
\[
mw \equiv 1 - v \pmod{p}.
\]
Choose $a \in (\xi + \idealP^w) \setminus (\xi + \idealP^{w+1})$, so $a = \xi + u \unif^w$ with some $\idealP$-unit $u \in \fieldKp$. Applying the binomial expansion, we can write $\pg(a) = \pg(\xi) + \unif^w\cdot c$ for some $c\in \valring$. Since $w > v$, we have $\ordp \pg(a) = v$. Therefore,
\[
\ordp \pf(a) 
= \ordp (a - \xi)^m + \ordp \pg(a)
= m\cdot w + v
\equiv 1\pmod{p}. 
\]
Thus, $\pf(a)$ cannot be a $p$-th power, contradicting $\pf \in \CKp$.
\end{proof}

The following result---in the spirit of Hensel's lemma---is likely known to experts in the field. However, lacking a convenient reference, we feel obliged to formulate and prove it explicitly.

\begin{lem}\label{lem:roots_in_all_quo}
Let $\pf\in \polyring$ be a polynomial. If for every $n\geq 1$ the polynomial~$\pf$ has a root in $\quo{\intring}{\idealP^n}$, then $\pf$ has a root in~$\valring$.
\end{lem}

\begin{proof}
Polynomials are continuous with respect to the $\idealP$-adic topology, and each $\idealP^n$ is a closed subset of~$\valring$. Hence,
\begin{equation}\label{eq:chain}
\pf^{-1}(\idealP) \supset \pf^{-1}(\idealP^2) \supset \dotsb 
\supset \pf^{-1}(\idealP^n) \supset \dotsb
\end{equation}
is a descending chain of closed subsets of~$\valring$. By assumption, for every $n$ there exists $a_n \in \quo{\intring}{\idealP^n}$ with $\pf(a_n) = 0$. Choose a lift $b_n\in \valring$ of~$a_n$. Then $\pf(b_n)\in \idealP^n$, so $\pf^{-1}(\idealP^n)$ is nonempty. Consequently, \eqref{eq:chain} is a descending chain of nonempty closed subsets of the compact set~$\valring$. Thus, it has a nonempty intersection. Pick $a \in \bigcap_{n \geq 1} \pf^{-1}(\idealP^n)$. Then $\pf(a) \in \idealP^n$ for every exponent $n\geq 1$, which means that $\pf(a) = 0$.
\end{proof}

\begin{prop}\label{prop:decidability_CZ}
Let $\pf \in \polyring$ be a polynomial with no roots in~$\fieldKp$, and assume that $\pf$ is not divisible by a $p$-th power of another polynomial. Then there is a finite, explicitly constructed set $\setA\subset \intring$ \textup(depending on~$\pf$\textup) such that
\[
\pf\in \COp \iff \pf\in \CZ \iff \pf\in \setC(\setA).
\]
\end{prop}

The equivalence of the first two conditions was already established in Proposition~\ref{prop:CK_eq_CKp}. It remains to prove the existence of a set~$\setA$ for which the other equivalence holds. This is done in the following algorithm.

\begin{alg}\label{alg:text_CZ}
Let $\pf \in \polyring$ be as in Proposition~\ref{prop:decidability_CZ}. This algorithm decides whether $\pf \in \CZ$.
\begin{enumerate}
\item Set $M\gets  \bigl\lfloor \sfrac{ep}{(p-1)}\bigr\rfloor + 1$, and initialize $m\gets 0$. 
\item\label{st:outter_loop} Let $\setA \subset \intring$ be a system of representatives of the quotient ring $\quo{\intring}{\idealP^{m + M}}$. 
\item For each $a \in \setA$, perform the following steps:
    \begin{enumerate}
    \item If $\pf(a) \notin \pow$ \textup(use, e.g., \cite[Algorithm~5]{KN2025}\textup), then output `false' and halt.
    \item\label{st:update_m} If $\ordp \pf(a) > m$, update $m \gets \ordp \pf(a)$ and return to step~\eqref{st:outter_loop}.
    \end{enumerate}
\item Output `true'.
\end{enumerate}
\end{alg}

\begin{poc}
We first prove that the algorithm terminates. To this end, it suffices to show that the test in step~\eqref{st:update_m} can be triggered only finitely many times. A contrario, suppose that this is not the case. Then there exists an infinite sequence $(a_i)_{i \in \NN}$ of elements of~$\intring$ such that the sequence $\bigl(\ordp \pf(a_i)\smash{\bigl)}_{i\in\NN}$ of valuations is strictly increasing (and hence unbounded). Fix any positive integer~$n$, and let~$i$ be an index such that 
\[
\ordp \pf(a_{i-1}) < n \leq \ordp \pf(a_i).
\]
Then the residue class of~$a_i$ is a root of~$\pf$ in the ring $\quo{\intring}{\idealP^n}$.  
This implies that $\pf$ has a root in every quotient ring $\quo{\intring}{\idealP^n}$, and therefore a root in~$\valring \subset \fieldKp$ by Lemma~\ref{lem:roots_in_all_quo}, contradicting the assumption on~$\pf$. This proves that the algorithm indeed terminates.

We now prove the correctness of the output. Clearly, the algorithm can return `false' only when it finds an explicit counterexample to the condition `$\pf \in \CZ$'. It remains to show that the algorithm does not return `true' incorrectly (i.e.\ that it produces no false positives).

Let $m_1, m_2, \dotsc$ denote the successive values of the variable~$m$ during the execution, and let $\setA_1, \setA_2, \dotsc$ be the corresponding sets of representatives. In particular, $\setA_i$ is a system of representatives of $\quo{\intring}{\idealP^{m_i + M}}$. 

Suppose there exists $b \in \intring$ such that $\pf(b) \notin \pow$. Then there exists $a_1 \in \setA_1$ such that 
\[
a_1\equiv b\pmod{\idealP^{m_1 + M}}.
\]
It follows that $\pf(a_1) \equiv \pf(b) \pmod{\idealP^{m_1 + M}}$.  

If $\ordp \pf(b) = m_1 = 0$, then also $\ordp \pf(a_1) = 0$, and Corollary~\ref{cor:equal_classes_suff} implies that $\pf(b) \cdot \unpow = \pf(a_1) \cdot \unpow$. Hence $\pf(a_1) \notin \pow$, and the algorithm correctly outputs `false'.

On the other hand, if $\ordp \pf(b)$ is greater than~$m_1$, then so is $\ordp \pf(a_1)$. In this case $\setA_2$ is a system of representatives of $\quo{\intring}{\idealP^{m_2 + M}}$ with $m_2 = \ordp \pf(a_1)$. Again, $b$ is congruent to some $a_2 \in \setA_2$ modulo $\idealP^{m_2 + M}$. Repeating this process, we eventually find $a_k \in \setA_k$ such that
\[
\pf(b) \equiv \pf(a_k) \pmod{\idealP^{m_k + M}}
\qquad\text{and}\qquad 
\ordp \pf(b) = \ordp \pf(a_k) \leq m_k.
\]
By Corollary~\ref{cor:equal_classes_suff}, the $p$-th power classes $\pf(b)\cdot \unpow$ and $\pf(a_k)\cdot \unpow$ coincide. Hence $\pf(a_k) \notin \pow$, and the algorithm outputs `false'. This shows that the algorithm never returns `true' incorrectly.

Thus, the algorithm terminates and is correct.
\end{poc}

\begin{rem}
Algorithm~\ref{alg:text_CZ}, as presented, is sufficient to establish the existence of~$\setA$. From a computational perspective, however, it is rather suboptimal. In particular, once an element~$a$ with $\ordp \pf(a) > m$ is found, it would be more economical to extend~$\setA$ only by those representatives of $\quo{\intring}{\idealP^{M + \ordp \pf(a)}}$ that are congruent to~$a$ modulo $\idealP^{M+m}$, rather than reconstructing the entire set. Furthermore, an optimized implementation would record which elements have already been tested, thereby avoiding redundant repetitions that may occur in the current version when $\setA$ is updated.
\end{rem}

Algorithm~\ref{alg:text_CZ} constructs the set~$\setA$ explicitly, but it does not yield any clear information about its size. Our next goal is therefore to derive an upper bound for $\card{\setA}$. To this end, we first recall the notion of the Krasner constant.

Let $\pf\in \fieldKp[x]$ be a polynomial of degree $d := \deg \pf \geq 2$, and let $\fieldL\supset \fieldKp$ be a splitting field of~$\pf$. Denote by $\ordp$ the (unique) extension of the valuation from $\fieldKp$ to~$\fieldL$. If $\xi_1, \dotsc, \xi_d\in \fieldL$ are all the roots of~$\pf$, then the \term{Krasner constant} of~$\pf$ is defined as
\[
\kras\pf := \max\bigl\{ \ordp(\xi_u - \xi_j)\st i\neq j \bigr\}.
\]

\begin{lem}\label{lem:bound_ord_Fa}
If a polynomial $\pf\in \polyring$ of degree~$d$ has no roots in~$\fieldKp$, then
\[
\max \bigl\{ \ordp \pf(a) \st a\in \valring\bigr\} <  d\cdot \kras\pf + \ordp \lc(\pf).
\]
\end{lem}

\begin{proof}
Denote $m := \max \bigl\{ \ordp \pf(a) \st a\in \valring\bigr\}$. Writing $\pf$ in factored form,
\[
\pf = \lc(\pf)\cdot (x - \xi_1)\dotsm (x - \xi_d).
\]
we obtain for any $a\in \valring$:
\begin{equation}\label{eq:krasner}
\ordp \pf(a) = \ordp \lc(\pf) + \sum_{i\leq d} \ordp (a - \xi_i).
\end{equation}

Now, the Krasner's lemma (see, e.g., \cite[Lemma~4.5]{Murty2002} or \cite[Proposition~4.4.6]{Cohen2007}) asserts that if for some $i\leq d$ we had
\[
\ordp(a - \xi_i) > \max_{j\neq i}\ordp(\xi_j - \xi_i),
\]
then $\fieldKp(\xi_i) \subseteq \fieldKp(a) = \fieldKp$, contradicting the assumption that $\pf$ has no root in~$\fieldKp$. It follows that for every $i\leq d$, we have 
\[
\ordp (a - \xi_i) \leq \max\{\ordp(\xi_j - \xi_i)\st j\neq i\} \leq \kras\pf.
\]
Therefore, \eqref{eq:krasner} yields the inequality $\ordp \pf(a) \leq \ordp \lc(\pf) + d\kras\pf$. Thus $m \leq \ordp \lc(\pf) + d\kras\pf$, as expected. 
\end{proof}

\begin{prop}
The cardinality of~$\setA$ is bounded above by $p^A$, where
\[
A := \frac{efp}{p-1} + fd\cdot \kras\pf + f\cdot \ordp\lc(\pf).
\]
\end{prop}

\begin{proof}
Denote $M := \sfrac{ep}{(p-1)}$ and $\kappa := \kras\pf$. By the way the set~$\setA$ is constructed in Algorithm~\ref{alg:text_CZ}, we have
\[
\card{\setA} = \card{\quo{\valring}{\idealP^{m + M}}} = p^{fm + fM},
\]
where
\[
m = \max\bigl\{ \ordp \pf(a) \st a\in \valring \bigr\}. 
\]
It therefore suffices to show that $m \leq d\kappa + \ordp \lc(\pf)$.

\end{proof}

At first sight, it may seem that we have simply replaced one unknown quantity depending on~$\pf$ (the size of~$\setA$) with another (the Krasner constant). However, the Krasner constant has been the subject of extensive study, and effective bounds are available in terms of the degree and the height of~$\pf$. In particular, \cite[Lemma~2.3]{Pejkovic2012} shows that if $\pf$ is a polynomial with coefficients in the ring~$\ZZ$ of rational integers, then
\[
\kras\pf \leq \frac32\cdot d\cdot \log_p(d)\cdot H(f)^{1-d},
\]
where $H(\pf)$ denotes the \term{height} of~$\pf$, i.e., the maximum absolute value of its coefficients.  

As a direct consequence, we obtain:

\begin{cor}
If $f\in \ZZ[x]$, then
\[
\log_p \card{\setA} \leq
\frac{p}{p-1} + \frac{3}{2}\cdot d^2\cdot \log_p(d)\cdot H(\pf)^{1-d} + \ord_p\lc(\pf).
\]
\end{cor}

For further results on the Krasner constant, we refer the reader to \cite{BD2025, BK2002, Khanduja1999, Khanduja2002}. We can now present an explicit algorithm for testing if a polynomial belongs to the class~$\CKp$.

\begin{alg}\label{alg:test_CK}
Given a polynomial $\pf\in \polyring$ this algorithm decides whether $\pf\in \CK$. 
\begin{enumerate}
\item Using polynomial square-free decomposition and clearing denominators, write
\[
c^p\cdot \pf = \lc(\pf)\cdot \tilde\pg_1\cdot \tilde\pg_2^2\dotsm \tilde\pg_k^k,
\]
where  $c \in \ZZ$ and $\tilde\pg_1, \dotsc, \tilde\pg_k\in \polyring$ are square-free polynomials.
\item Set
\[
\pf_* := \lc(\pf) \cdot \prod_{i=1}^k \tilde\pg_i^{i \bmod p}.
\]
\item Check if $\pf_*$ has a root in~$\fieldKp$, if it does output `false' and halt. 
\item Execute Algorithm~\ref{alg:text_CZ} for $\pf_*$ and $\rev \pf_*$. 
\item If the result is `false' for any of these two polynomials, output `false'. Otherwise, output `true'.
\end{enumerate}
\end{alg}

The correctness of the above algorithm is immediate in view of the preceding discussion and Proposition~\ref{prop:CK_iff_F_and_revF}.

\begin{rem}
One may also consider a broader class of polynomials than the one studied in this paper. Fix a subset 
\[
\setS = \bigl\{s_1\unpow, \dotsc, s_k\unpow\bigr\}
\]
of the $p$-power classes, and define the class of polynomials that map the entire field~$\fieldKp$ into the union of these classes (possibly together with zero):
\[
\bigl\{ \pf\in \polyring \st 
\pf(a) \in \{0\}\cup s_1\unpow\cup \dotsb \cup s_k\unpow 
\;\text{ for all }a\in \fieldKp \bigr\}.
\]
For instance, the polynomial $27x^9 + 54x^6 + 54x^3 + 40$ maps $\QQ_3$ into $\QQ_3^{\times 3}\cup 4\QQ_3^{\times 3}$. In terms of the illustration from the introduction, such polynomials only \emph{partially} fold the `fan' of $p$-power classes, rather than collapsing it completely.  

Easy adaptations of Algorithms~\ref{alg:text_CZ} and~\ref{alg:test_CK} provide a procedure for testing if a polynomial belongs to these more general classes. Unfortunately, these classes no longer form monoids.
\end{rem}

\section{Relations to the $p$-th powers}\label{sec:Pth_powers}
We can now focus on the original question---posed in the introduction---whether the polynomials that assume $p$-th powers everywhere on~$\fieldKp$ must already be $p$-th powers of some other polynomials. We show that it is not the case (see Proposition~\ref{prop:non_power_poly}). Moreover, in general, such polynomials are not even approximated by $p$-th powers except on a compact subset of~$\fieldKp$ (see Proposition~\ref{prop:approximation} and the comment following it). The first observation is trivial:

\begin{obs}
For every subset $\setA\subset \fieldKp$, the set $\setC(\setA)$ is a multiplicative monoid containing the submonoid $\polyring\cap \fieldKp[x]^p$.
\end{obs}

The original question mentioned in the introduction asked whether this submonoid is proper when $\setA = \QQ$ and $\fieldKp = \QQ_2$. The following proposition shows that the answer is affirmative for every number field~$\fieldK$.

\begin{prop}\label{prop:non_power_poly}
The monoid $\polyring\cap \fieldKp[x]^p$ is a proper submonoid of $\CK$.
\end{prop}

\begin{proof}
We must show that $\CK$ contains polynomials that are not $p$-th powers of other polynomials. As always, let $\unif\in \intring$ be a uniformizer of~$\idealP$. Consider the following polynomial:
\[
\pf := \bigl(1 + \unif\cdot x^p)^p + \unif^m\in \polyring,
\]
where $m$ is an integer strictly greater than $\sfrac{ep}{(p-1)}$. We claim that $\pf(a) \in \pow$ for every $a\in \fieldK$. 

Write $a$ in the form $a = \unif^v\cdot u$, where $u$ is a $\idealP$-unit and $v = \ordp a$. We shall consider two cases. First, assume that $v \geq 0$. Then, $1 + \unif\cdot a$ is a $\idealP$-adic unit and so $\pf(a)$ is congruent to its $p$-th power modulo~$\idealP^m$. By Proposition~\ref{prop:local_power_thm}, it follows that $\pf(a)$ is a $p$-th power.

Conversely, assume that $v < 0$. It is clear that $\pf(a)$ belongs to~$\pow$ if and only if $\unif^{-vp^2 - p}\cdot \pf(a)$ does. Write
\[
\unif^{-vp^2 - p}\cdot \pf(\unif^v\cdot u)
= \bigl(\unif^{-vp-1} + u^p\bigr)^p + \unif^{m - vp^2 - p}.
\]
Since $v < 0$, we have $-vp - 1 > 1$ and $m - vp^2 - p > m$. Consequently,
\[
\ordp (\unif^{-vp-1} + u^p) = \min \{ -vp - 1, 0\} = 0.
\]
Therefore, the expression in parentheses is a $\idealP$-adic unit, and the entire right-hand side is congruent to a $p$-th power of a unit modulo $\idealP^{m - vp^2 - p}\subset \idealP^m$. Thus, $\pf(a)\in \pow$, as claimed.

It remains to show that $\pf$ is not itself a $p$-th power of a polynomial over~$\fieldKp$. Differentiating, we obtain
\[
\pf' = \unif\cdot p^2\cdot x^{p - 1}\cdot \bigl(1 + \unif\cdot x^p\bigr)^{p-1}.
\]
Moreover,
\[
\pf \equiv 1 + \unif^m\neq 0\pmod{x}
\qquad\text{and}\qquad 
\pf \equiv \unif^m\neq 0\pmod{ 1 + \unif\cdot x^p}. 
\]
It follows that $\pf$ and $\pf'$ are relatively prime, and hence $\pf$ is square-free. A square-free polynomial cannot be a perfect $p$-th power, which completes the proof.
\end{proof}

Let us now return to the analogy between the class~$\CKp$ and the class of real positive semi-definite polynomials. Let $\|\cdot \|$ denote the max-norm on the space of real polynomials:
\[
\| \pf_0 + \pf_1x + \dotsb + \pf_dx^d \| := \max\bigl\{ |\pf_1|, \dotsc, |\pf_d| \bigr\}.
\]
Consider a real polynomial $\pf\in \RR[x]$ that is strictly positive on~$\RR$. Then there exists $\varepsilon > 0$ such that every polynomial~$\pg$ satisfying $\|\pf - \pg\| < \varepsilon$ is also strictly positive on~$\RR$. This stability property fails, however, when $\pf$ is merely non-negative. For instance, $\pf = x^2$ can be approximated by $\pg = x^2 - \xi^2$ with $\xi$ arbitrarily small, yet $\pg$ is no longer non-negative. An analogous phenomenon occurs for the class~$\CKp$. 

\begin{prop}\label{prop:continuity}
Let $\pf = \pf_0 + \pf_1x + \dotsb +\pf_dx^d$ be a polynomial of degree~$d$. Assume that $f\in \CKp$ and $\pf$ has no roots in~$\fieldKp$. Let $M$ be an integer satisfying
\[
M\geq d\cdot \kras\pf + \ordp(\pf_0\cdot \pf_d) + \frac{ep}{p-1}.
\]
If $\pg = \pg_0 + \pg_1x + \dotsb + \pg_dx^d$ is another polynomial and
\[
\min\bigl\{ \ordp(\pf_i - \pg_i) \st i\leq d\bigr\} > M,
\]
then $\pg\in \CKp$.
\end{prop}

\begin{proof}
Write $\pg$ as $\pg = \pf + \ph$, where $\ph = \ph_0 + \ph_1x + \dotsb + \ph_dx^d$, By assumption, $\ordp\ph_i > M$ for every $i\leq d$. Take any $a\in \valring$. Lemma~\ref{lem:bound_ord_Fa} implies that
\begin{multline*}
\ordp\ph(a)
= \ordp \Bigl(\sum_{i\leq d} \ph_i a^i\Bigr)
\geq \min\bigl\{ \ordp\ph_i + i\cdot \ordp a\st i\leq d\bigr\}
> M \geq\\
\geq d\cdot \kras\pf + \ordp \pf_d + \frac{ep}{(p-1)}
> \ordp \pf(a) + \frac{ep}{(p-1)}.
\end{multline*}
Consequently, $\ordp\pg(a) = \ordp \pf(a)$ for all $a\in\valring$. Hence, $\pg(a)\in \pf(a)\cdot \unpow = \unpow$ by Corollary~\ref{cor:equal_classes_suff}. Thus, we have $\pg\in \COp$.

Consider now the reciprocal polynomial of~$\pg$. We have
\[
\rev \pg = \rev\pf + \rev \ph, 
\]
and the same arguments as above (except that this time we use $\ordp\pf_0$ instead of $\ordp\pf_d$) show that $\rev\pg\in \COp$. By Proposition~\ref{prop:CK_iff_F_and_revF}, we conclude that $\pg\in \CKp$.
\end{proof}

As in the real case, the assumption that $\pf$ has no roots in~$\fieldKp$ is indispensable.
Indeed, the polynomial $\pf(x)=x^p$ belongs to~$\CKp$, yet it can be approximated by polynomials of the form $\pg(x)=(x-\xi_1)\dotsm(x-\xi_p)$ with $\xi_i$ sufficiently close to zero. It is obvious that one can choose $\xi_1, \dotsc, \xi_p$ in such a way that $\pg$ does not lie in $\CKp$.

\begin{rem}
In particular, Proposition~\ref{prop:continuity} implies that the monoid~$\CKp$  is considerably larger than $\polyring\cap \fieldKp[x]^p$.
\end{rem}

Finally, to complete the picture, we show that polynomials from~$\CZ$ can be approximated over~$\valring$ by $p$-th powers of polynomials. However, polynomials from~$\CK$ cannot be approximated this way on the whole~$\fieldKp$ unless they are already $p$-th powers themselves.

\begin{prop}\label{prop:approximation}
Let $\pf\in \CZ$. Then for every positive integer~$n$, there is a polynomial $\pg\in \polyring$ such that
\[
\ordp (\pf - \pg^p)(a) \geq n
\]
for all $a\in \valring$.
\end{prop}

One possible way to prove this proposition is to take a continuous function $\Gamma: \valring\to \fieldKp$ such that 
\[
\pf(a) = \Gamma(a)^p
\]
for all $a\in \valring$. Then apply Mahler’s theorem (see, e.g., \cite[Theorem~1.1]{Shalit2016}) to approximate~$\Gamma$ by a polynomial~$\pg$. Alternatively, there is a short and direct proof, which we present below.

\begin{proof}
\newcommand{\bpf}{\overline{\pf}}
\newcommand{\bpg}{\overline{\pg}}
Fix $n\geq 0$ and consider the quotient ring $R_n := \quo{\intring}{\idealP^n}$. Let $\bpf\in R_n[x]$ be the image of~$\pf$ under the canonical epimorphism $\polyring\onto R_n[x]$. By assumption, $\pf$ belongs to $\CZ$. Hence, for every $a\in R_n$ there exists $b\in R_n$ such that $\bpf(a) = b^p$. The function mapping~$a$ to~$b$ is polynomial by \cite{Frisch1999}. Thus, there is $\bpg\in R_n[x]$ such that $\bpf - \bpg^p$ vanishes on the entire ring~$R_n$. Lift $\bpg$ to a polynomial $\pg\in \polyring$. Then $(\pf - \pg^p)(a)\in \idealP^n$ for every $a\in \valring$, which concludes the proof.
\end{proof}

It goes without saying that the above approximation of~$\pf$ by a $p$-th power of another polynomial is possible only over~$\valring$ (or other compact subsets), but not on the whole field~$\fieldKp$. IThis follows from the well-known fact that any polynomial bounded on an entire local field must be constant. Indeed, fix a polynomial $\pf\in \CK$ and an integer~$n$. Suppose that there is some polynomial $\pg\in \polyring$ such that 
\begin{equation}\label{eq:approx}
\ordp (\pf - \pg^p)(a) \geq n
\end{equation}
for all $a\in \fieldKp$. Set $\ph := \pf - \pg^p$. If $\ph$ is not constant, say $\ph = \ph_0 + \ph_1x + \dotsb + \ph_d x^d$ with $\ph_d\neq 0$ and $d > 0$, choose any $a\in \fieldKp$ satisfying
\[
\ordp a <
\frac{1}{d + 1}\cdot \bigl(
\min\{ n, \ordp \ph_0, \dotsc, \ordp \ph_{d-1}\} 
- \ordp \ph_d\bigr).
\]
Then $\ordp \ph(a) = \ordp \ph_da^d < 0$, contradicting~\eqref{eq:approx}.

\bibliographystyle{plain}
\bibliography{C-polys}

\end{document}